\documentclass{amsart}

\usepackage[latin1]{inputenc}
\usepackage{subfigure,amsmath,graphicx,color}

\usepackage{mathptmx}

\newtheorem{theorem}{Theorem} %[section]
\newtheorem{lemma}[theorem]{Lemma}
\newtheorem{proposition}[theorem]{Proposition}
\newtheorem{corollary}[theorem]{Corollary}

\newcommand{\bN}{\mathbb{N}}

\newcommand{\bR}{\mathbb{R}}

\newcommand{\bH}{\mathbb{H}}

\newcommand{\bV}{\mathbb{\hspace{-.07mm}V\hspace{-.2mm}}}

\newcommand{\cF}{\mathcal{F}}

\newcommand{\cH}{\mathcal{H}}

\newcommand{\cL}{\mathcal{L}}

\newcommand{\cV}{\mathcal{V}}

\newcommand{\unif}{\text{\rm unif}}
\newcommand{\var}{\text{\rm var}}

\newcommand{\Beta}{\text{\rm Beta}}
\newcommand{\GEM}{\text{\rm GEM}}

\newcommand{\dequ}{=_{\text{\rm\tiny d}}}
\newcommand{\TPL}{\text{\rm TPL}}
\newcommand{\HPL}{\text{\rm HPL}}
\newcommand{\WI}{\text{\rm WI}}

\newcommand{\RT}{\text{\rm RT}} 

\allowdisplaybreaks

\begin{document}

\title[Random recursive trees]{Random recursive trees: A boundary theory approach}

\author{Rudolf Gr\"ubel}
\address{Institut f\"ur Mathematische Stochastik\\ 
         Leibniz Universit\"at Hannover\\
         Postfach 6009\\
         30060 Hannover\\
         Germany}
\email{rgrubel@stochastik.uni-hannover.de}

%    author two information
\author{Igor Michailow}
\curraddr{}
\email{michail@stochastik.uni-hannover.de}

\date{\today}

\begin{abstract} 
  We show that an algorithmic construction of sequences of recursive trees leads to a direct proof of the convergence of
  random recursive trees in an associated  Doob-Martin compactification; it also gives a representation of the limit in terms of
  the input sequence of the algorithm. We further show that this approach can be used to obtain strong limit theorems
  for various tree functionals, such as path length or the Wiener index.
\end{abstract}

\maketitle

\section{Introduction}\label{sec:intro}
A tree with node set $[n]:=\{1,\ldots,n\}$ is recursive if the node numbers along the unique path from $1$ to $j$
increase for $j =2,\ldots,n$. Trees with this property may be encoded by a sequence $(j_1,\ldots,j_{n-1})$, where
$j_{k}\in [k]$ denotes the direct ancestor of $k+1$ (next node on the way to the `root' $1$). Such a sequence also gives
a recipe for growing the corresponding tree: Starting with the unique recursive tree of size (number of nodes) $1$,
which consists of the root node $1$ only, we obtain the respective next tree by joining node $k$ to node $j_{k-1}$,
$k=2,\ldots,n$.  Choosing the ancestor of the next node uniformly at random among the nodes of the current tree we obtain 
a sequence $Y_1,Y_2,\ldots$ of random recursive trees, which we collect into a stochastic process $Y=(Y_n)_{n\in\bN}$.

A survey of random recursive trees and their applications is given in~\cite{SmMah}; for a more recent reference
see~\cite[Chapter~6]{Drmota09}. Various functionals of these structures have been considered by different authors, a
representative but not exhaustive list being node degrees~\cite{RRTdegree,RRTdegr2,RRTdegr3}, height~\cite{RRTheight},
path length~\cite{Mah1991,RRTpathlength}, profiles~\cite{RRTprofile}, spectra~\cite{RRTspec}, and various `topological' indices,
such as the Wiener and the Zagreb indices~\cite{RRTWiener,RRTZagreb}. Often the results are limit theorems, with
(strong) convergence of the random variables or convergence of their distributions as $n\to\infty$. This, in the authors'
view, naturally raises the question of convergence of the trees themselves, with the aim of developing a systematic
approach to the strong asymptotics of tree functionals. The Doob-Martin compactification, initiated by the fundamental 
paper~\cite{Doob59}, is a general tool that can be used in this context; see~\cite{Woess2} for a recent textbook
introduction. In particular, using concepts from discrete potential theory it provides an enlargement of the state space 
of a Markov chain such that the variables converge almost surely. This approach has been used in ~\cite{EGW1} to obtain
convergence results for a class of randomly growing discrete structures that includes various random trees.

If we use the encoding explained in the opening paragraph then it is possible to retrace the full sequence $Y_1,\ldots,Y_{n-1}$ of
previous trees from the current tree $Y_n$. In such a case the discrete potential theory approach leads to convergence
in the sense of projective (or inverse) limits, which is of little help for proving convergence of functionals. Noting
that the functionals of interest are often invariant under relabelling (a phrase that has to be made precise) we
therefore choose a model that is coarser in the sense that it `forgets the labels' but retains the Markov property.
This partial loss of information turns the sequence $Y$ into a sequence $X=(X_n)_{n\in\bN}$ of randomly growing subsets
of a fixed infinite tree. For this chain, the Doob-Martin compactification has been determined in~\cite{EGW1}.  The
first of our aims here is to show that the convergence result provided by the general theory can be obtained more
directly by using a suitable algorithmic construction, and that this approach has the advantage of leading to a
description of the limit $X_\infty$ in terms of the input sequence of the algorithm. The
representation serves as the basis for the analysis of tree functionals such as different notions of path length and the
Wiener index; indeed, our second objective is to obtain strong limit theorems for such functionals. A similar strategy
has been used in~\cite{GrMtree} for binary search trees.

In the next section we first take care of a variety formal details, including some terminology and notation, and
then give a new `constructive' proof of the basic limit result. In Section~\ref{sec:func} we discuss various
tree functionals and comment on the connections to related work. 

\section{The limit tree and its distribution}\label{sec:limit}
We introduce Harris trees and the Harris chain generated by the RRT process; in view of its confounding
potential we spell out the details of the transition from recursive to Harris trees. From the RRT sequence
Harris chains inherit a useful decomposition property. Next, we recall from~\cite{EGW1} the Doob-Martin
compactification of the Harris chain. Then we explain an algorithm which is then used to give a new proof of
that part of~\cite[Theorem~6.1]{EGW1} that is relevant for our present purposes, together with a
representation of the limit. Finally, we collect some auxiliary results on the distribution of the limit that will be useful in the next
section when we analyze tree functionals.

\subsection{From recursive trees to Harris trees}\label{subsec:rec2Harris} 
We regard the set $\bV=\bN^\star$ of finite sequences of natural numbers as the set of potential tree nodes and write
$u+v=(u_1,\ldots,u_k,v_1,\ldots,v_l)$ for the concatenation of the nodes $u=(u_1,\ldots,u_k)$ and $v=(v_1,\ldots,v_l)$,
abbreviating $u+(i)$ to $ui$, $i\in\bN$. By a \emph{Harris tree} we mean a finite
subset $x$ of $\bV$ with the properties
\begin{itemize}
\item[(H1)] if $u+v\in x$, then $u\in x$,
\item[(H2)] if $ui\in x$ with $i>1$, then $uj\in x$ for $j=1,\ldots,i-1$.
\end{itemize}
Condition (H1) is prefix stability if we regard nodes as words with letters from the alphabet $\bN$. In a family tree
interpretation, condition (H2) means that a non-root node must either be the first child of its ancestor node or that it
must have earlier-born siblings. Harris trees are also known as Ulam-Harris trees; they may be seen as rooted planar
trees with a specific labelling of nodes.

We write $\bH$ for the set of Harris trees and $\bH_n$ for the subset of those trees that have $n$ nodes. In order to
relate Harris trees to recursive trees we map the nodes $j$ of a recursive tree to words $u(j)=(u_1(j),\ldots,u_k(j))\in\bV$
as follows: The length $k$ of the word is the distance to the root of (the node labelled) $j$, and $u_k(j)$ is the number of nodes
$i\in [j]$ that have the same direct ancestor as $j$.  The prefix sequences similarly encode the nodes from the root to
$j$. This corresponds to an embedding of recursive trees into the plane where new nodes are placed to the right of their
siblings.  

%\bigbreak
\begin{figure}%[h]
  \begin{center}
  \phantom{a}\vspace{.5cm}
  \includegraphics[width=12cm]{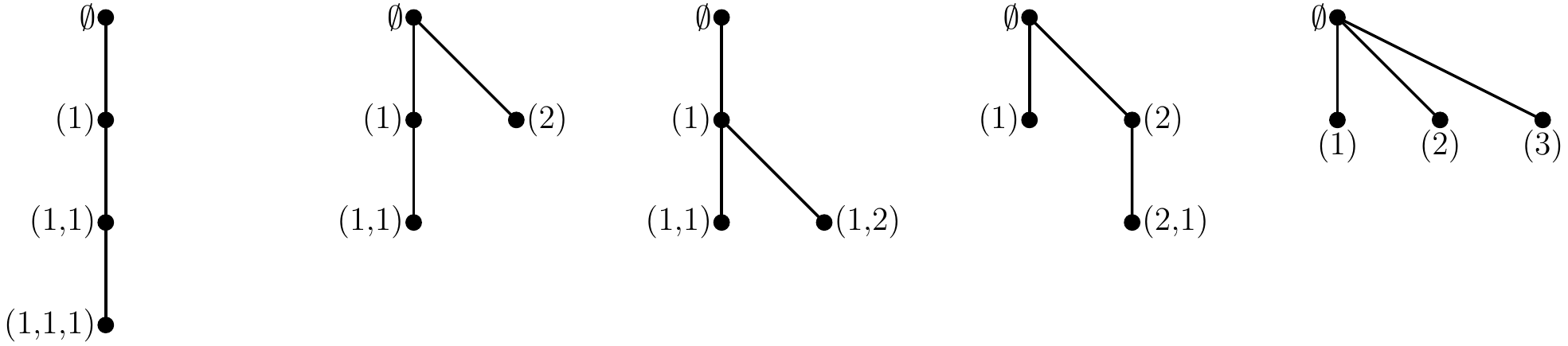}
 \end{center}
\vspace{-.3cm}
\caption{The Harris trees with four nodes.}\label{fig:fournodes}
\end{figure}

Clearly, there are $(n-1)!$ possibilities for the encoding sequences for recursive trees with $n$ nodes, 
hence this is also the number of recursive trees with $n$ nodes.  
Figure~\ref{fig:fournodes} shows the five elements of $\bH_4$. Of the $(4-1)!=6$ recursive trees with four
nodes, encoded by $(1,2,3)$, $(1,1,2)$, $(1,2,1)$, $(1,2,2)$, $(1,1,3)$ and $(1,1,1)$ respectively, the second and  third are
mapped to the same Harris tree. The figure also offers an opportunity to comment on the informal expression of
`forgetting the labels' that we used above and that often appears in the literature: It is tempting to regard this as
passing from graphs to isomorphism classes, but this is not what is happening here---indeed, the second and the fourth
Harris tree in Figure~\ref{fig:fournodes} are isomorphic as rooted trees. A compatible notion of equivalence and
isomorphism in the present situation can be obtained on the basis of the above planar embedding of recursive
trees. 

Writing $\Psi$ for the function that maps recursive trees to Harris trees, we define $X=(X_n)_{n\in\bN}$ by
$X_n:=\Psi(Y_n)$ for all $n\in\bN$, where $Y$ is the RRT chain introduced in Section~\ref{sec:intro}. In the original
process, $Y_n$ is uniformly distributed on its range, but $X_n$ is not uniformly distributed on $\bH_n$ as explained
above for $n=4$.  In this new process, it is no longer possible to `trace back' to previous values. As $\Psi$ does
not change the number of nodes it is adapted to the combinatorial family $\bH$ in the sense that $P(X_n\in\bH_n)=1$ for
all $n\in\bN$. To see that it retains the Markov property and to obtain the corresponding transition probabilities we
argue as follows: Let $y_n,y_n'$ be recursive trees with $n$ nodes and let $y_{n+1}$ be a recursive tree with $n+1$
nodes. Suppose that $\Psi(y_n)=\Psi(y_n')=:x_n$ and let $x_{n+1}:=\Psi(y_{n+1})$. If $x_n\subset x_{n+1}$ then there is
a unique recursive tree $z_{n+1}$ such that $\Psi(z_{n+1})=x_{n+1}$ and
\begin{equation*}
  P(X_{n+1}=x_{n+1}|Y_n=y_n)\, = \, P(Y_{n+1}=z_{n+1}|Y_n=y_n)\, = \, \frac{1}{n}, 
\end{equation*}
and similarly there is a $z_{n+1}'$ with the same property for $y_n'$. Clearly, if $x_n\not\subset x_{n+1}$, then
these probabilities will be 0.  This shows that
\begin{equation*}
    P(X_{n+1}=x_{n+1}|Y_n=y_n)\, = \,   P(X_{n+1}=x_{n+1}|Y_n=y_n')
\end{equation*}
whenever $\Psi(y_n)=\Psi(y_n')$, and further that
\begin{equation*}
    P(X_{n+1}=x_{n+1}|X_n=x_n)\, = \,
             \begin{cases}  1/n, & x_n\subset x_{n+1},\\ 0, &\text{otherwise}, \end{cases}
\end{equation*}
for $x_n\in\bH_n$, $x_{n+1}\in\bH_{n+1}$. By~\cite[Lemma~2.5]{LPW} the first of these implies that $X$ is a Markov
chain; the second shows that, as with $Y$, we select the ancestor for the new node uniformly at random in the step
from $X_n$ to $X_{n+1}$. 

\subsection{A tree decomposition}\label{subsec:treedecomp}
We associate with a node $u=(u_1,\ldots,u_k)\in\bV$ its `flat' and `raised' version
\begin{equation*}
 u^\flat :=\; (1,u_1,\ldots,u_k), \quad 
      u^\sharp := \begin{cases} (1+u_1,u_2,\ldots,u_k)&\text{if } u\not=\emptyset,\\ 
                                            \quad\emptyset,&\text{if } u=\emptyset, 
                       \end{cases}
\end{equation*}
and lift this to trees $x\in\bH$ via
\begin{equation*}
  x^\flat := \{ u\in\bV:\, u^\flat \in x\},\quad x^\sharp :=\; \{ u\in\bV:\, u^\sharp \in x\}.
\end{equation*}
These are the subtree of $x$ rooted at $(1)$ and the shifted tree that remains if this subtree is taken out.
It is well known that the random variables $K_n:=\#X_n^\flat$, $X_n^\flat$ and $X_n^\sharp$ are independent,
with $K_n$ uniformly distributed on $[n-1]$, and that, conditionally on $K_n=k$, $X_n^\flat$ and $X_n^\sharp$ have
the same distribution as $X_k$ and $X_{n-k}$ respectively. An interesting combinatorial proof of the
corresponding statement for the $Y$ process, based on a bijection between permutations and random recursive
trees, is given in~\cite{RRTpathlength}. An alternative proof can be obtained on using the algorithmic
background to be given in Section~\ref{subsec:alg} below.

\subsection{The Doob-Martin compactification of the Harris chain}\label{subsec:DMHarris} 
The paths of the stochastic process $X$ are sequences of growing subset of the set $\bV$ of all potential
nodes. We may regard $\bV$ itself as the infinite Harris tree (note that this tree is not locally finite). It
can be shown that, in the $X$-sequence, every potential node will eventually be an element of the infinite 
Harris tree. Hence,
if we embed $\bH$ into $\{0,1\}^\bV$ via the node indicators,
\begin{equation*}
  x\,\mapsto\, \bigl(u\mapsto 1_x(u)\bigr),
\end{equation*}
then $X_n$ converges almost surely to this infinite tree, which is represented by the function that is constant~1. This, however,
does not capture the `true' asymptotics of~$X$. In contrast, Markov chain boundary theory
provides a state space completion (compactification) $\bar \bH$ of $\bH$ with the properties
\begin{itemize}
\item[(L)] $X_n\to X_\infty\in\partial\bH:=\bar\bH\setminus\bH$ with probability~1 as $n\to\infty$,
\item[(T)] $X_\infty$ generates the tail $\sigma$-field associated with $X$, up to null sets. 
\end{itemize}
For (T), we require that the chain has the space-time property, meaning that the time parameter $n$ is a function of the
state $x$.  For the Harris  sequence this is the case, so (T) implies that the Doob-Martin compactification captures the
persisting randomness of the sequence, whereas for any one-point compactification the $\sigma$-field generated by the
limit will always be trivial in the sense that only 0 and 1 arise as probabilities of its elements.

The Doob-Martin compactification $\bar \bH$ of $\bH$ with respect to the Harris chain has been identified
in~\cite{EGW1}. Let
\begin{equation*}
        \bar\bV\; :=\;  \bN^\star \sqcup \bN^\infty\,\sqcup\, \bigsqcup_{k=0}^\infty \bN^k\times \{\infty\}^\infty.
\end{equation*}
In words: $\bar\bV$ consists of all finite and infinite sequences of natural numbers, plus all infinite sequences
$u=(u_i)_{i\in\bN} \subset \bN\sqcup\{\infty\}$ with the property that, for some $k\in\bN$, $u_i\in\bN$ for $i<k$
and $u_i=\infty$ for $i\ge k$. For $u\in\bV$ and $v\in\bar\bV$ write $u\le v$ if $u$ is a prefix of $v$ and put
\begin{equation*}
  A_u:=\{v\in\bar\bV:\, u\le v\}, \quad u\in\bV.
\end{equation*}
Let $\cV$ be the $\sigma$-field on $\bar\bV$ generated by the sets $A_u$, $u\in\bV$, let $\bar\bH$ be the set of 
probability measures $\mu$ on $(\bV,\cV)$, and endow $\bar\bH$ with the coarsest topology that makes the 
functions $\mu\mapsto \mu(A_u)$, $u\in\bV$, continuous. Finally, embed $\bH$ into $\bar\bH$ by identifying
$x\in\bH$ with the uniform distribution on $x$ as a subset of $\bV$.
Then $\bar\bH$ is the Doob-Martin compactification of $\bH$ induced by the chain $X$, up to homeomorphism. 

\subsection{The algorithmic construction}\label{subsec:alg}
For $x\in\bH$ and $u\in\bV$ let $x(u):=\{v\in\bV:\, u+v\in x\}$ be the subtree of $x$ rooted at $u$. Then the embedding
of $\bH$ into $\bar\bH$ may be written as
\begin{equation*}
  x \, \mapsto\, \bigl(u \mapsto \# x(u)/\#x\bigr),  
\end{equation*}
and we can restate the
convergence in the Doob-Martin topology of a sequence $(x_n)_{n \in\bN}$ with $x_n\in\bH_n$ for all $n\in\bN$  
to $\mu\in \partial\bH$  as
\begin{equation*}
  \lim_{n\to\infty} \frac{1}{n}\, \# x_n(u) = \mu(A_u)\quad\text{for all } u\in\bV.
\end{equation*}
Our plan is to prove the almost sure convergence of the Harris chain $X$ in this topology by using an algorithm that
generates $X$ if the input is chosen appropriately. 

The \emph{recursive tree algorithm} maps an input sequence $t=(t_n)_{n\in\bN}$ of pairwise distinct positive real
numbers to an output sequence $(x_n,\phi_n)_{n\in\bN}$ of labelled trees, with $x_n\in\bH_n$ and $\phi_n:x_n\to\bR$. The
algorithm works sequentially, starting with $x_1=\{\emptyset\}$ and the label $\phi_1(\emptyset)=t_0:=0$ for the root
node. As explained in Section~\ref{sec:intro} and at the end of Section~\ref{subsec:rec2Harris}, we need to specify the
direct ancestor of the new node $v$ to be added in the step from $x_n$ to $x_{n+1}$: We attach $v$ as a next (resp.~the
first) child to the node with label $\max\{t_j:\, j=0,\ldots,n-1, \, t_j<t_n\}$ and then label $v$ by $t_n$.
Figure~\ref{fig:RRT} shows an example where new children are positioned to the right of their older siblings. By
$\RT(t)$ we mean the sequence $(x_n)_{n\in\bN}$, i.e.~we ignore the labels.

\begin{figure}%[h]
  \begin{center}
  \includegraphics[width=5cm]{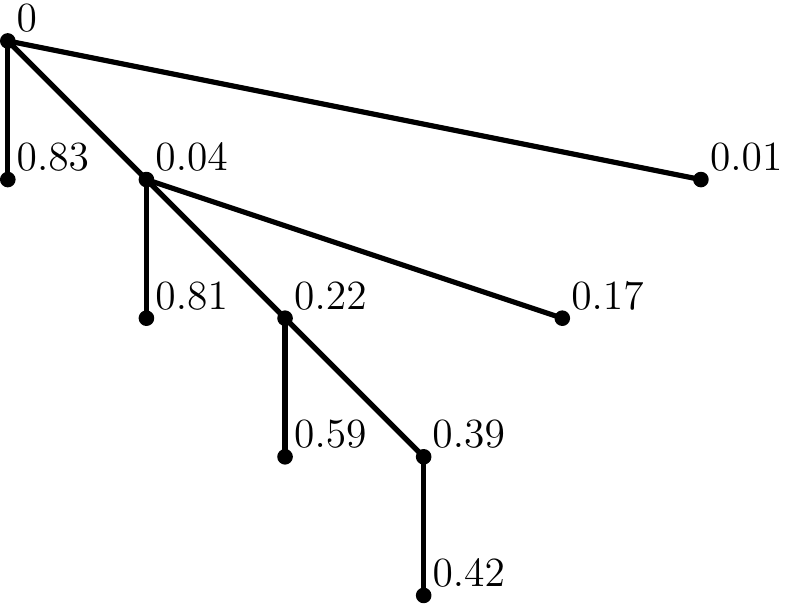}
 \end{center}
\caption{The tree obtained from $\ t=(.83,.04,.81,.22,.59,.01,.39,.42,.17,\ldots)$.}\label{fig:RRT}
\end{figure}

Clearly, if the trees converge, then the limit must be a function of the input sequence. In order to be able to specify
this relationship we need some more notation: Given an increasing sequence $(x_n)_{n\in\bN}$ of Harris trees, let
\begin{equation*}
  \tau(u) \, :=\,  \inf\{n\in\bN:\, u\in x_{n+1}\} 
\end{equation*}
(note that $x_n$ is built from $t_1,\ldots, t_{n-1}$). 
Further, for any sequence $(t_n)_{n\in\bN}$ of pairwise distinct elements of the open unit interval let 
\begin{equation*}
  0=:t_{(n:0)}< t_{(n:1)}<t_{(n:2)}<\cdots<t_{(n:n)}<t_{(n:n+1)}:=1
\end{equation*}
be the (augmented) increasing order statistics associated with the first $n$ values $t_1,\ldots,t_n$, and let
\begin{equation*}
    \kappa(u) :=\#\{1\le i\le \tau(u): \, t_i\le t_{\tau(u)}\}  
\end{equation*}
be the rank of $t_{\tau(u)}$ in $t_1,\ldots,t_{\tau(u)}$, so that $t_{(\tau(u):\kappa(u))}=t_{\tau(u)}$. In
Figure~\ref{fig:RRT} for example, the node $u=(2,2)$ has $\tau(u)=4$, $\kappa(u)=2$ and $t_{\tau(u)}=0.22$.

The following result relates the algorithm and the limit object. Let $\unif(0,1)$ be the uniform
distribution on the unit interval. We write $\cL(Y)$ for the distribution (law) of the random quantity $Y$ and sometimes
use $Y\sim \mu$ instead of $\cL(Y)=\mu$.

\begin{theorem}\label{thm:lim1} 
Let $\eta_i$, $i\in\bN$, be independent random variables, with $\eta_i\sim\unif(0,1)$ for  all $i\in\bN$. 

\emph{(a)} The algorithm $\RT$ generates the RRT chain in the sense that $\RT(\eta)$ and $X$ are identical in distribution. 

\smallbreak
\emph{(b)} Suppose that $X=\RT(\eta)$. Then 
$X_n$ converges almost surely to $X_\infty$ in the Doob-Martin topology as
$n\to\infty$, where  on a set of probability~1 the limit $X_\infty$ is given by
\begin{equation}\label{eq:Xinfty}
  X_\infty(A_u) = \eta_{(\tau(u):\kappa(u)+1)} -\eta_{\tau(u)} \quad \text{for all }u\in\bV.
\end{equation}
\end{theorem}

\begin{proof} 
Part (a) belongs to the folklore of the subject.  Due to its
importance for the present paper we recall for the proof that the rank of $\eta_n$ in $\eta_1,\ldots,\eta_n$ is
uniformly distributed on $\{1,\ldots,n\}$, and that rank $i$ means that $\eta_n$ is attached to the
node with label $\eta_{(n:i-1)} $ (which is the root if $i=1$).

With each node $u$ we associate the interval $I(u)=(\eta_{\tau(u)},\eta_{(\tau(u):\kappa(u)+1)})$. From the definition of
the RT algorithm, nodes added to the tree at a time $n>\tau(u)$ will have prefix $u$ if and only if $\eta_n\in I(u)$. 
The random variables $\eta_{\tau(u)+n}$, $n\in\bN$, are independent and uniformly distributed on the unit interval,
hence~\eqref{eq:Xinfty} follows with the Glivenko-Cantelli theorem. 
\end{proof} 

Theorem~\ref{thm:lim1} can be related to the corresponding result~\cite[Theorem~1]{GrMtree} for binary search trees via the natural or rotation
correspondence between Harris trees and binary trees~\cite[Section 2.3.2]{Knuth1}~\cite[p.73]{FlSedge}; details are given in~\cite{IMdiss}.

In addition to the convergence of the trees we also obtain the distribution of the limit $X_\infty$, which takes its values in 
the set of probability measures $\mu$ on $(\bar\bV,\cV)$. As a preliminary step we extend the tree decomposition
introduced in Section~\ref{subsec:treedecomp} to $\bar\bH$ as follows: For $\mu\in\partial\bH$
with $0<\mu(A_{(1)})<1$ we define $\mu^\flat,\mu^\sharp\in\partial \bH$ by 
\begin{equation*}
   \mu^\flat(A_u) = \frac{\mu(A_{u^\flat})}{\mu(A_{(1)})}, \ \ 
   \mu^\sharp(A_u)=\frac{\mu(A_{u^\sharp})}{1-\mu(A_{(1)})}
\end{equation*}
for all $u\in\bV\setminus\{\emptyset\}$, and $\mu^\flat (A_u)=\mu^\sharp (A_u)=1$ if $u=\emptyset$. 

\begin{proposition}\label{prop:decompinfty}
Let $X_\infty$ be as in Theorem~\ref{thm:lim1}. Then the random variables $\eta:=X_\infty(A_{(1)})$,
$X_\infty^\sharp$ and $X_\infty^\flat$ are independent. Further, $\eta\sim\unif(0,1)$,
and $X_\infty^\sharp$ and $X_\infty^\flat$ have the same distribution as $X_\infty$.
\end{proposition}

\begin{proof}
Let $\eta=(\eta_i)_{i\in\bN}$ be a sequence of independent, $\unif(0,1)$-distributed random variables. We define
two new sequences $\eta^\flat=(\eta^\flat_i)_{i\in\bN}$ and $\eta^\sharp=(\eta^\sharp_i)_{i\in\bN}$ by successively
transforming  the $\eta_i$'s with $\eta_i>\eta_1$ into $\eta_j^\flat=(\eta_i-\eta_1)/(1-\eta_1)$ and the $\eta_i$'s
with $\eta_i<\eta_1$ into $\eta^\sharp_j=\eta_i/\eta_1$. Clearly, $\eta_1$, $\eta^\flat$ and $\eta^\sharp$ are 
independent, and $\eta^\flat$ and $\eta^\sharp$ are again sequences of independent, $\unif(0,1)$-distributed 
random variables. From this, the statement of the theorem follows in view of $X_\infty(A_{(1)})=1-\eta_1$, 
$X^\flat=\RT(\eta^\flat)$, and $X^\sharp=\RT(\eta^\sharp)$.  
\end{proof}

We call $\mu$ atom-free and diffuse if
\begin{equation*}
  \mu(\{u\})=0\quad \text{and} \quad \mu(A_u)>0\quad\text{for all } u\in\bV.
\end{equation*}
Let $\Sigma_\infty\subset [0,1]^\infty$ be the infinite-dimensional probability simplex, that is, the set of all sequences
$(\rho_i)_{i\in\bN}$ with $\rho_i\ge 0$ for all $i\in\bN$ and $\sum_{i=1}^\infty \rho_i=1$. An atom-free and diffuse $\mu$ 
associates with each $u\in\bV$ an element $\rho(\mu,u)=(\rho_i(\mu,u))_{i\in\bN}$ of $\Sigma_\infty$ via
\begin{equation*} 
         \rho_i(\mu,u):=\frac{\mu(A_{ui})}{\mu(A_u)}\quad \text{for all } i\in\bN. 
\end{equation*}
For later use we note that, for such $\mu$,
\begin{equation*}
  \rho(\mu,u^\sharp) = \rho(\mu^\sharp,u), \quad \rho(\mu,u^\flat) = \rho(\mu^\flat,u)\quad\text{for all } u\in\bV. 
\end{equation*}
Clearly, $\mu$ can be reconstructed form $\rho(\mu,u)$, $u\in\bV$.  In fact,
\begin{equation}\label{eq:fromrho2mu}
  \mu(A_u) = \prod_{i=1}^k \rho_{u_i}\bigl(\mu,(u_1,\ldots,u_{i-1})\bigr)\quad\text{for all } u=(u_1,\ldots,u_k)\in\bV. 
\end{equation}
Of course, for random input both the $\tau$- and the $\kappa$-values will be random, as will be $\mu$. 

We say that a random variable $\xi=(\xi_i)_{i\in\bN}$ 
with values in $\Sigma_\infty$ has the (standard) GEM (Griffiths-Engen-McCloskey) distribution if its components can be written as
\begin{equation}\label{eq:xidef1}
  \xi_1=\zeta_1,\quad \xi_i=\zeta_i\,\prod_{j=1}^{i-1}(1-\zeta_j)\ \text{ for } i>1,
\end{equation}
with $\zeta_i$, $i\in\bN$, independent and $\zeta_i\sim\unif(0,1)$ for all $i\in\bN$.

At each level $k$, the sets $A_{u}$ with $|u|=k$ provide a partition of $\bar\bV\setminus \bN^{k-1}$. The corresponding values
$X_\infty(A_u)$ are related to the $k$th nested decomposition of the unit interval into descending records of the input
sequence. This interpretation suggests the following result, which gives a description of the distribution of $X_\infty$.

\begin{theorem}\label{thm:lim2} 
Let $X_\infty$ be as in Theorem~\ref{thm:lim1}. Then the random variables $\rho(X_\infty,u)$, $u\in\bV$, are independent 
and GEM distributed.
\end{theorem}

\begin{proof} 
By Proposition~\ref{prop:decompinfty}, $\rho_1(X_\infty,\emptyset)=1-\eta_1=X_\infty(A_{(1)})$, $X_\infty^\flat$ and
$X_\infty^\sharp$ are independent.  Repeating the decomposition with the respective raised part, we obtain that the
variables 
\begin{equation*}
     \frac{\rho_i(X_\infty,\emptyset)}{\rho_{i-1}(X_\infty,\emptyset)}, \quad i\in\bN,
\end{equation*}
with $\rho_o(X_\infty,\emptyset):=1$, are independent and $\unif(0,1)$-distributed, that these 
are independent of the random probability measures $X_{\infty,i}$, $i\in\bN$, defined by
\begin{equation*}
  X_{\infty,i}(A_u) := \frac{X_\infty(A_{(i)+u})}{X_\infty(A_{(i)})}, \quad u\in\bV,
\end{equation*}
and that the these measures are independent and identical in distribution to $X_\infty$. (It is easy to see that
$X_{\infty,i}$ arises as the $\flat$-part of the $i$th iteration of the decomposition) . In particular,
$\rho(X_\infty,\emptyset)\sim\GEM$. Taken together, this proves the case $k=0$ of the
following statement:
\begin{itemize}
\item[(i)] $\rho(X_\infty,u)\sim\GEM$ for all $u\in\bV$ with $|u|\le k$,
\item[(ii)] the random sequences $\rho(X_\infty,u)$,  $u\in\bV$, $|u|\le k$, are independent,
\item[(iii)] the random measures $X_{\infty,v}$, $v\in\bV$, $|v|=k+1$, given by
\begin{equation*}
  X_{\infty,v}(A_u) := \frac{X_\infty(A_{v+u})}{X_\infty(A_v)}, \ u\in\bV,
\end{equation*}
are independent and identical in distribution to $X_\infty$,
\item[(iv)] $\{\rho(X_\infty,u):\, u\in\bV,|u|\le k\}$ and $\{X_{\infty,v}:\, v\in\bV, |v|=k+1\}$ are
independent.
\end{itemize}
We can apply the same reasoning used for $k=0$ separately to each of the nodes at level $k+1$
to obtain the induction step from $k$ to $k+1$.

This shows that the above compound statement holds for all $k\in\bN$; clearly, (i) and (ii) imply
the assertion of the theorem.
\end{proof} 

In view of the fact that $X_\infty(A_u)$ is a function of the variables  $\rho(X_\infty,v)$ with $|v|<|u|$ we obtain
that $X_\infty(A_u)$ and $\rho(X_\infty,u)$ are independent, for all $u\in\bV$.

\subsection{Conditional distributions}\label{subsec:condexp}
In order to be able to use the general limit theorem for the analysis of tree functionals in the next section
we need the conditional distribution of $X_\infty$ given $X_n$. For this we rely on the results
in~\cite[Section~6]{EGW1}; we also need some more notation.

The distribution $\Beta(\alpha,\beta)$ with parameters $\alpha,\beta>0$ is given by its density
\begin{equation}\label{eq:betadens}
    f(t|\alpha,\beta) \, = \, \frac{\Gamma(\alpha+\beta)}{\Gamma(\alpha)\Gamma(\beta)}\,
                                                       t^{\alpha-1}(1-t)^{\beta-1},\quad 0<t<1.
\end{equation}
For later use we recall that
\begin{equation} \label{eq:mombeta}
  E\xi=\frac{\alpha}{\alpha+\beta},\quad E\xi^2=\frac{\alpha(\alpha+1)}{(\alpha+\beta)(\alpha+\beta+1)} 
                                \quad\text{if }\xi\sim\Beta(\alpha,\beta), 
\end{equation}
and, clearly, $\Beta(1,1)=\unif(0,1)$. For
$a=(a_1,\ldots,a_k)\in\bN^\star$ we write $\GEM(a)$ for the distribution of the $\Sigma_\infty$-valued random 
sequence $\xi=(\xi_i)_{i\in\bN}$ given by
\begin{equation}\label{eq:eta2xi}
  \xi_1=\zeta_1, \quad \xi_i=\zeta_i\prod_{j=1}^{i-1} (1-\zeta_j)\quad \text{for  } i > 1,
\end{equation}
where $\zeta_i$, $i\in\bN$, are independent and 
\begin{equation}\label{eq:etadistr}
  \cL(\zeta_i) =  \begin{cases}
                                \Beta\bigl(a_i,1+\sum_{j=i+1}^k a_j\bigr), &\text{for } i<k,\\ 
                                \Beta(a_k,1), &\text{for } i=k,\\ 
                                \Beta(1,1), &\text{for } i>k. \end{cases}
\end{equation}
Interestingly, the marginals of such random sequences are again beta distributed (of course, they are no longer
independent). 

\begin{lemma}\label{lem:GEMmargin}
If $\xi=(\xi_i)_{i\in\bN}\sim\GEM(a)$ with $a=(a_1,\ldots,a_k)\in\bN^\star$, then, with $b:=\sum_{i=1}^ka_i$, 
\begin{equation*}
  \xi_i\sim\Beta(a_i,1+b-a_i)\ \text{ for }\;  i=1,\ldots,k.
\end{equation*}  
Moreover, with $(\zeta_i)_{i\in\bN}$ as in~\eqref{eq:eta2xi} and~\eqref{eq:etadistr}
\begin{equation*}
            1- \sum_{j=1}^k \xi_j\, =\, \prod_{j=1}^k(1-\zeta_j)\sim\Beta(1,b).
\end{equation*}
\end{lemma}

\begin{proof}
This follows with the known rule for products of independent
beta-distributed random variables, see e.g.~\cite[p.378, Exercise~11.8]{KSO}.
\end{proof}

Recall that the distribution of $X_\infty$ is specified by the (joint) distribution of the $\Sigma_\infty$-valued
quantities $\rho(X_\infty,u)$, $u\in\bV$, and that $\#x(ui)>0$ implies $\#x(uj)>0$ for $j=1,\ldots,i-1$ by property~(H2)
of Harris trees; see also~\eqref{eq:fromrho2mu}.

\begin{theorem}\label{thm:conddistr}
The conditional distribution of $\rho(X_\infty,u)$ given $X_n$ is $\GEM(a)$, where $a=(a_1,\ldots,a_k)$ with
\begin{equation}\label{eq:GEMpar}
    k=\max\{i\in\bN: \, \#X_n(ui)>0\}, \quad a_i=\#X_n(ui)\ \text{ for }i=1,\ldots,k.
\end{equation}
Further, the random sequences $\rho(X_\infty,u)$, $u\in\bV$, are conditionally independent given $X_n$.
\end{theorem}

\begin{proof}
By the general theory of Markov chain boundaries the conditional distribution $Q_2$ of $X_\infty$ given 
$X_n=x\in\bH_n$ has density $K(x,\cdot)$ with respect to the (unconditional) distribution $Q_1$ of $X_\infty$, 
where $K$ denotes the extended Martin kernel. Note that $Q_1$ and $Q_2$ are probability measures on
the set $\bar\bH$ of probability measures $\mu$ on $(\bar\bV,\cV)$, where the $\sigma$-field on $\bar\bH$ is the
one generated by the evaluation maps $\mu\mapsto \mu(A_u)$, $u\in\bV$.  
The extended Martin kernel has been determined in~\cite{EGW1}: It can be written 
as the product of `local extended kernels',
\begin{equation}\label{eq:Kprod}
  K(x,\mu) = \prod_{u\in\bV} K_u\bigl(u(x),\rho(\mu,u)\bigr),
\end{equation}
where $u(x)=(a_1,\ldots,a_k)\in\bN^\star$ is given by~\eqref{eq:GEMpar} with $x$ instead of $X_n$, and
\begin{equation}\label{eq:Kfact}
  K_u(x,s)\ = \ \frac{\bigl(\sum_{i=1}^k a_i\bigr)!}{\prod_{i=1}^k(a_i-1)!} \; \prod_{i=1}^k s_i^{a_i-1}
                           \; \prod_{i=1}^{k-1}\Bigl(1-\sum_{j=1}^i s_i\Bigr)
\end{equation}
for all $s=(s_i)_{i\in\bN}\in\Sigma_\infty$. Also, $K_u(x,\cdot)$ is the conditional density of the distribution
$Q_{2,u}$ of $\rho(X_\infty,u)$ given $X_n=x$ with respect to its corresponding unconditional
counterpart~$Q_{1,u}$, which we know to be the GEM distribution. The product form~\eqref{eq:Kprod} implies that the
independence of the sequences $\rho(X_\infty,u)$, $u\in\bV$, remains intact in the transition from $Q_1$ 
to $Q_2$. This proves the second part of the theorem. 

Now let $T:\Sigma_\infty\to [0,1]^\infty$ be given by
\begin{equation*}
  (s_i)_{i\in\bN} \mapsto (t_i)_{i\in\bN}, \quad t_i:=\frac{s_i}{1-s_1-\cdots -s_{i-1}} \ \text{ for all } i\in\bN.
\end{equation*}
This is the inverse of the transition from $\zeta$ to $\xi$ in~\eqref{eq:eta2xi}. We know that the push-forward 
$Q_{1,u}^T$ of $Q_{1,u}$ under $T$ is the infinite product of uniforms. The first part of the 
theorem refers to the push-forward $Q_{2,u}^T$ of $Q_{2,u}$ under $T$; it asserts that a density
of $Q_{2,u}^T$ with respect to $Q_{1,u}^T$ is given by
\begin{equation*}
 g(t) \; = \; \prod_{i=1}^{k-1} f\Bigl(t_i\Big|a_i,\sum_{j=i+1}^k a_j\Bigr) \cdot f(t_k|a_k,1)
\end{equation*}
for almost all $t=(t_i)_{i\in\bN}\in [0,1]^\infty$, with $f$ as in~\eqref{eq:betadens}. 
With all this notation in place it remains to check that $g\circ T=K_u(x,\cdot)$, with $K_u$ as in~\eqref{eq:Kfact}. 
This, however, is a bookkeeping task.
\end{proof}

The embedding of $\bH$ into $\bar\bH$, which maps $X_n$ to the uniform distribution on its nodes, leads to an
interpretation of $X_n$ as a real-valued random function on $\bV$ via $u\mapsto \#X_n(u)/n$. Similarly, the limit
$X_\infty$ can be seen as the random function $u\mapsto X_\infty(A_u)$ on $\bV$. Obviously, all these functions are
bounded and, if we endow $\bV$ with the discrete topology, they are continuous. This displays $X_n$, $n\in\bN$, and
$X_\infty$ as random elements of an infinite-dimensional separable Banach space.

\begin{corollary} \label{cor:toThmconddistr}
Let $(\cF_n)_{n\in\bN}$, with $\cF_n:=\sigma(X_1,\ldots,X_n)$, $n\in\bN$, be the natural filtration 
of the Harris chain $(X_n)_{n\in\bN}$. Then 
\begin{equation*}
  X_n=E[X_\infty|\cF_n]\quad\text{for all } n\in\bN.
\end{equation*}
In particular, $(X_n,\cF_n)_{n\in\bN}$ is a martingale.  
\end{corollary}

\begin{proof}
We have $E[X_\infty|\cF_n]=E[X_\infty|X_n]$ due to the Markov property. Further, the notion of infinite-dimensional martingale,
see e.g.~\cite[Section-V.2]{NeveuMart}, in the present context means that we have to check that
\begin{equation*}
  E[X_\infty(A_u)|X_n] = \frac{1}{n}\,\# X_n(u)\quad\text{for all } u\in\bV.
\end{equation*}
Let $u=(u_1,\ldots,u_k)\in\bV$ be given and let 
\begin{equation*}
  \xi_i := \rho_{u_i}\bigl(X_\infty,(u_1,\ldots,u_{i-1})\bigr),\quad i=1,\ldots,k.
\end{equation*}
From~\eqref{eq:fromrho2mu} we obtain $X_\infty(A_u)=\prod_{i=1}^k \xi_i$, and by Theorem~\ref{thm:conddistr} the factors 
are conditionally independent given $X_n$. Hence, using Lemma~\ref{lem:GEMmargin} and~\eqref{eq:mombeta},
\begin{align*}
  E[X_\infty(A_u)|X_n] \ &=\ \prod_{i=1}^k E[\xi_i|X_n] \\
          &=\ \prod_{i=1}^k \frac{\# X_n\bigl((u_1,\ldots,u_i)\bigr)}
                               {1+\sum_{j=1}^\infty \#X_n\bigl((u_1,\ldots,u_{i-1},j)\bigr)} \\
         &=\ \prod_{i=1}^k \frac{\# X_n\bigl((u_1,\ldots,u_i)\bigr)}
                               {\#X_n\bigl((u_1,\ldots,u_{i-1})\bigr)} \quad
        =\ \frac{1}{n}\, \#X_n(u).  \qedhere
\end{align*}
\end{proof}

This result may be seen as a consequence of the general Doob-Martin construction. 

\section{Tree functionals}\label{sec:func}
Let $Y=(Y_n)_{n\in\bN}$ be the RRT chain and let $X=(X_n)_{n\in\bN}$, with $X_n=\Psi(Y_n)$ for all $n\in\bN$, be the
associated Harris chain. In this section we consider functionals of the recursive trees that are invariant under
$\Psi$ and hence can be written as functions $V_n=\Phi(X_n)$ of the $X$-variables. A typical example is the
total path length, which is the sum of the depth of all nodes in the tree. The methods discussed below can be 
applied to fairly general functions~$\Phi$, but here we will restrict ourselves to the real-valued case.

There are two main probabilistic methods to obtain distributional or even strong limit results for suitably standardized
versions of the $V$-variables. In the first of these, we try to find a suitable martingale and then apply a martingale
limit theorem. In the second, we use the internal structure of the $X$-variables to find a recursion for the
$V$-variables and then apply Banach's fixed point theorem with a suitably chosen metric space of probability 
measures. The prototypical example is the number of comparisons needed by the Quicksort algorithm,
which can be related to the total path length of binary search trees: The martingale approach is carried out 
in~\cite{RegnierQS}, whereas~\cite{RoeQS} employed the second approach, which since then has come to be 
known as  the contraction method. The two methods may fruitfully be combined, as exemplified by~\cite{RRTpathlength}
in connection with the total path length of random recursive trees. 

On its own the martingale method does not say anything about the limit, and the contraction method may miss the fact
that the random variables themselves converge. The method suggested in the present paper and in~\cite{GrMtree} needs
some additional investment in connection with proving the convergence of the discrete structures themselves but then
provides a unifying approach: In view of the fact that $X_\infty$ generates the tail $\sigma$-field associated with the
Harris chain, see property (T) in Section~\ref{subsec:DMHarris}, any almost sure limit $Y_\infty$ must be a functional
$Y_\infty=\Psi(X_\infty)$ of $X_\infty$, up to null sets. Projecting $Y_\infty$ on the natural filtration we obtain a
convergent martingale, which often turns out to be a simple transformation of the variables $V_n$. Below we carry this
out for two versions of the total path length and for the Wiener index.

\subsection{Total path length}\label{subsec:TPL}
This is simply the sum of all node depths and can be written in terms of subtree sizes as
\begin{equation*}
  \TPL(x) \; :=\;  \sum_{u\in x} |u|\;   =\;  \sum_{u\in x} \#x(u) -\#x, \quad x\in\bH.
\end{equation*}
Here we have written $|u|$ for the length (or depth) $k$ of $u=(u_1,\ldots,u_k)\in\bV$.
We need the auxiliary function 
\begin{equation*}
  C: \Sigma_\infty\to [-\infty,1], \quad (s_i)_{i\in\bN}\mapsto 1 + \sum_{i=1}^\infty s_i\log s_i.
\end{equation*}
The harmonic numbers 
\begin{equation*}
  H_0:=1 , \quad H_n:=\sum_{k=1}^n \frac{1}{k} \ \text{ for all } n\in\bN,
\end{equation*}
will appear repeatedly;  we will write $H(n)$ instead of $H_n$ whenever this is typographically more convenient.
We collect some auxiliary statements.

\begin{lemma}\label{lem:TPL1} \emph{(a)} If $\xi\sim \GEM(a)$ for some $a\in\bN^\star$ then $\|C(\xi)\|_p<\infty\;$ for all $p\ge 1$.
  
\smallbreak
\emph{(b)} If $\xi\sim \GEM(a)$ with $a=(a_1,\ldots,a_k)\in\bN^\star$, then
\begin{equation}\label{eq:EWC}
  EC(\xi) \; = \; 1\, + \, \frac{\sum_{i=1}^k a_i H(a_i)}{1 + \sum_{i=1}^k a_i}\, - \, H\Bigl(1+\sum_{i=1}^k a_i\Bigr).
\end{equation}
In particular, $EC(\xi)=0\,$ if $\xi\sim \GEM$. 
\end{lemma}

\begin{proof} For the proof of the first part we assume that $a=\emptyset$ and use the representation 
of $\xi$ by a sequence $(\zeta_i)_{i\in\bN}$ of independent random
variables with distribution $\unif(0,1)$, see~\eqref{eq:eta2xi}. Then, for each $i\in\bN$,
\begin{align*}
  \|\xi_i\log\xi_i\|_p
       \ &=\ \Bigl\| \Bigl(\zeta_i\prod_{j=1}^{i-1}(1-\zeta_j)\Bigr)\Bigl(\log\zeta_i +
                                                        \sum_{k=1}^{i-1}\log(1-\zeta_k)\Bigr)\Bigr\|_p\\
       &\le\ \Bigl\|\zeta_i(\log\zeta_i)\prod_{j=1}^{i-1}(1-\zeta_j)\Bigr\|_p \; +\; 
                         \sum_{k=1}^{i-1}\Bigl\|\zeta_i \log(1-\zeta_k)\prod_{j=1}^{i-1}(1-\zeta_j)\Bigr\|_p \\
       &=\ \|\zeta_i\log\zeta_i\|_p\prod_{j=1}^{i-1}\|1-\zeta_j\|_p\\
       &\hspace{1.5cm}                \; +\; \sum_{k=1}^{i-1}\|\zeta_i\|_p\|(1-\zeta_k)\log(1-\zeta_k)\|_p
                                           \prod_{j\in [i-1]\setminus\{k\}}\|1-\zeta_j\|_p\\
      &=\ \|\zeta_1\log\zeta_1\|_p\|\zeta_1\|_p^{i-1}
                              \; +\; (i-1)\|\zeta_1\|_p\|\zeta_1\log\zeta_1\|_p\|\zeta_1\|_p^{i-2}, 
\end{align*}
where we have used independence and $\cL(\zeta_i)=\cL(1-\zeta_i)=\cL(\zeta_1)$. In view of 
\begin{equation*}
  \int_0^1|t^p(\log t)^p|\, dt<\infty, \qquad \|\zeta_1\|_p<1,
\end{equation*}
this shows that  $\|\xi_i\log\xi_i\|_p$ decreases at an exponential rate as $i\to\infty$.  The generalization to an
arbitrary $a\in\bN^\star$ is straightforward.

For the proof of (b) we first note that, for $\zeta\sim\Beta(i,j)$ with $i,j\in\bN$, 
\begin{equation}\label{eq:logbeta}
   E\bigl(\zeta\log(\zeta)\bigr) \; =\; \frac{i}{i+j}\, \bigl(H_i-H_{i+j}\bigr).
\end{equation}
Suppose now that $\xi\sim\GEM(a)$ with
$a=(a_1,\ldots,a_k)\in\bN^\star$ and let $b:=\sum_{j=1}^k a_j$. We have $\xi_i\sim\Beta(a_i,1+b-a_i)$
for  $j=1,\ldots,k$ by Lemma~\ref{lem:GEMmargin},  hence 
\begin{equation*}
  E\xi_i\log \xi_i = \frac{a_i}{1+b}\, \bigl(H(a_i)-H(1+b)\bigr), \quad  i=1,\ldots,k.
\end{equation*}
Using the second part of Lemma~\ref{lem:GEMmargin} we see that for $i>k$ we may write $\xi_i=\alpha_i\beta_i$ 
with $\alpha_i$ and $\beta_i$ independent, $\alpha_i\sim\Beta(1,b)$ and $\beta_i$ the product of 
$i-k$ independent $\unif(0,1)$-distributed random variables. This gives, using~\eqref{eq:logbeta} again,
\begin{align*}
  E\xi_i\log \xi_i\ &=\ E\alpha_i \; E\beta_i\log\beta_i\; + \; E\beta_i \; E\alpha_i\log \alpha_i\\
      &=\ \frac{1}{1+b} \,\frac{i-k}{2^{i-k-1}}\, \frac{(-1)}{4}
                \;+\;   \frac{1}{1+b} \bigl(H(1)-H(1+b)\bigr)\, \frac{1}{2^{i-k}},
\end{align*}
so that, after some elementary manipulations,
\begin{equation*}
  \sum_{i=k+1}^\infty E\xi_i\log \xi_i\; =\; -\,\frac{H(1+b)}{1+b}.
\end{equation*}
Putting pieces together we finally arrive at~\eqref{eq:EWC}.
\end{proof}

Let 
\begin{equation}\label{eq:l1norm}
  |u|_1:=\sum_{i=1}^k u_i\quad \text{for all } u=(u_1,\ldots,u_k)\in\bV.
\end{equation}

\begin{lemma}\label{lem:TPL2} For $u\in\bV$ with $|u|_1=k$,
  \begin{equation*}
    X_\infty(A_u) \,\dequ\, \prod_{i=1}^k \zeta_i,
 \end{equation*}
with $\zeta_1,\ldots,\zeta_k$ independent, $\zeta_i\sim\unif(0,1)$ for $i=1,\ldots,k$.
\end{lemma}

\begin{proof} With each $u=(u_1,\ldots,u_l)\in\bV$ we associate its direct predecessor respectively direct elder
sibling $\bar u$ by
\begin{equation*}
  \bar u:=
  \begin{cases}
    (u_1,\ldots,u_{l-1}), &\text{if } u_l=1,\\ (u_1,\ldots,u_{l-1},u_l-1), &\text{if } u_l>1.
  \end{cases}
\end{equation*}
We may then connect the root $\emptyset=:u[0]$ to $u[k]:=u$ with nodes $u[i]$, $i=1,\ldots,k-l$ in such a way
that $u[i-1]=\bar u[i]$ for $i=1,\ldots,k$. The ratios $X_\infty(A_{u[i]})/X_\infty(A_{u[i-1]})$ are independent and
$\unif(0,1)$-distributed, by~\eqref{eq:xidef1} for a step to the right and by Theorem~\ref{thm:lim2} 
for a down-step.
\end{proof}

The transition $u\mapsto \bar u$ in the proof corresponds to the
transition to the direct ancestor (next node on the path to the root)
in the infinite binary tree $\{0,1\}^\star$ associated with $\bV$ by
the natural correspondence mentioned after the proof of 
Theorem~\ref{thm:lim1}.

\begin{lemma}\label{lem:TPL3} The sequence $(Y_{\infty,k})_{k\in\bN}$ with
\begin{equation*}
  Y_{\infty,k} := \sum_{u\in\bV,|u|_1\le k} X_\infty(A_u)\,C\bigl(\rho(X_\infty,u)\bigr)\quad\text{for all } k\in\bN,
\end{equation*}
converges in $L^p$ for all $p\ge 1$. 
\end{lemma}

\begin{proof} Let $p>1$.
We introduce the local abbreviations 
\begin{equation*}
  \bV[k]:=\{u\in\bV:\, |u|_1=k\}, \quad \cH_k:=\sigma(\{X_\infty(A_u):\, u\in\bV[k]\}).
\end{equation*}
Then 
\begin{equation*}
  Y_{\infty,k}-Y_{\infty,k-1}\; =\; \sum_{u\in\bV[k]} X_\infty(A_u)\, C(\rho(X_\infty,u)).
\end{equation*}
Lemma~\ref{lem:TPL2} yields 
\begin{equation}\label{eq:lemTPL2bound}
  E\bigl(X_\infty(A_u)\bigr)^p\; = \; \frac{1}{(1+p)^k}
\end{equation}
for all $u\in\bV[k]$.  On $X_\infty(A_u)=\alpha(u)$, $u\in\bV[k]$, we have
\begin{equation*}
  \cL\bigl(Y_{\infty,k}-Y_{\infty,k-1}\big|\cH_k\bigr)\; 
                =\; \cL\Bigl(\sum_{u\in\bV[k]}\alpha(u)\, \zeta_u\Bigr),
\end{equation*}
with $\zeta_u$, $u\in\bV[k]$, independent and identically distributed; further, 
$E|\zeta_u|^p<\infty$ by part~(a) of Lemma~\ref{lem:TPL1}.
Rosenthal's inequality, see e.g.~\cite[p.59]{Petrov}, gives
\begin{equation*}
  E\Bigl|\sum_{u\in\bV[k]}\alpha(u)\, \zeta_u\Bigr|^p
    \ \le \ c_p\,\biggl( \sum_{u\in\bV[k]}E\bigl|\alpha(u)\, \zeta_u\bigr|^p \; +\;  
                  \Bigl(\sum_{u\in\bV[k]}\var\bigl(\alpha(u)\, \zeta_u\bigr)\Bigr)^{p/2}\biggr)
\end{equation*}
with some constant that depends on $p$ only. 
Unconditioning and~\eqref{eq:lemTPL2bound} lead to upper bounds for both sums that decrease 
at an exponential rate $\kappa^k$ for some $\kappa<1/2$. This offsets the cardinality $2^k$ of $\bV[k]$,
and we conclude that $(Y_{\infty,k})_{k\in\bN}$ is a Cauchy sequence in $L^p$.
\end{proof}

Let
\begin{equation}\label{eq:Yinfty}
   Y_\infty:= \sum_{u\in\bV} X_\infty(A_u) \,C\bigl(\rho(X_\infty,u)\bigr)
\end{equation}
be the limit in Lemma~\ref{lem:TPL3}. 

\begin{theorem}\label{thm:TPL} 
As $n\to\infty$,
\begin{equation*}
  \frac{1}{n}\,\TPL(X_n)-H_n  +1 \;\to\; Y_\infty,  
\end{equation*}
almost surely and in $L^p$ for every $p>0$.  
\end{theorem}

\begin{proof} 
We project the prospective limit on the natural filtration introduced in Corollary~\ref{cor:toThmconddistr}: 
Using the remark after Theorem~\ref{thm:conddistr} and 
Lemma~\ref{lem:TPL1} we obtain
  \begin{align*}
    E[Y_\infty|\cF_n] \ &=\ \sum_{u\in\bV} E[X_\infty(A_u)|\cF_n]\, E[C(\rho(X_\infty,u))|\cF_n]\\
           &=\ \sum_{u\in X_n} \frac{\# X_n(u)}{n} 
                       \biggl(1+\frac{\sum_{i=1}^\infty \#X_n(ui)H(\#X_n(ui))}{\#X_n(u)} - H(\#X_n(u))\biggr)\\
          &=\ \frac{1}{n}\, \bigl(\TPL(X_n) + n\bigr) \, 
                             -\, \frac{1}{n}\, \sum_{u\in X_n}\Bigl(\#X_n(u) H(\#X_n(u)) - \sum_{i=1}^\infty
                             \#X_n(ui)H(\#X_n(ui))\Bigr)\\
          &=\ \frac{1}{n}\, \TPL(X_n)  + 1  - H(n),
\end{align*}
where a telescope effect simplified the sums. The statement of the theorem now follows with
the well-known martingale convergence theorems; see e.g.~\cite[Theorem IV-1-2, Proposition IV-2-7]{NeveuMart}.
\end{proof}

It is easy to see that $EY_\infty=0$, hence it follows from the calculations in the proof that the mean of the total
path length is given by $E\TPL(X_n)= nH_n-n$ for all $n\in\bN$.

The formula for the mean and the almost sure and $L^p$-convergence, $p>0$, of the standardized total path length 
of random recursive trees have
already been obtained in~\cite{Mah1991} and~\cite{RRTpathlength} respectively; we augment this by the representation of
the limit variable in terms of Doob-Martin limit $X_\infty$. The technical difficulty in the proof of almost sure
and $L^2$-convergence in~\cite{Mah1991}, as in its analogue for search trees in~\cite{RegnierQS}, consists of showing
that the respective martingales (which have to be found first) are bounded in $L^2$. Here we obtain the martingale
as a projection of a variable with finite second (or $p$th) moment onto the natural filtration of the Harris chain, which
implies the desired boundedness by Jensen's inequality for conditional expectations.

\subsection{Horizontal total path length}\label{subsec:HTPL}
We may regard the depth $|u|$ of a node $u$ as its vertical position; it is the number of downward moves (if this is
the direction of tree growth, from ancestor to child in familial terms) on the way from the root to $u$.  The (vertical)
total path length of a tree, considered in Section~\ref{subsec:TPL},  is the sum of these positions, taken over all
nodes in the tree. By the horizontal position of $u$
we mean the number of moves to the right (if this is where new nodes are added to an existing family) on the way from
the root to $u$. In the Harris encoding of nodes the horizontal position of the node $u=(u_1,\ldots,u_k)$ is given by
$|u|_1-|u|$, and the horizontal total path length of a tree is the sum of these positions over all
nodes of the tree,
\begin{equation*}
  \HPL(x)  := \sum_{u\in x} \bigl(|u|_1-|u|\bigr), \quad x\in\bH.
\end{equation*}
The horizontal position of a node can be seen as a recursive tree analogue of the notion of vertical position in a
binary tree; see~\cite[Chapter~5]{Drmota09} for the latter. The total horizontal path length does 
not seem to have been considered before, but a close relative is the total path degree length investigated
in~\cite{SzyTDPL}.  

We proceed as in the previous section, now using the auxiliary function 
\begin{equation*}
  D: \Sigma_\infty\to [-2,\infty], \quad (s_i)_{i\in\bN}\mapsto -2 + \sum_{i=1}^\infty i \, s_i.
\end{equation*}
For $\xi=(\xi_i)_{i\in\bN} \sim \GEM$ the representation~\eqref{eq:xidef1} implies
\begin{equation*}
   E\xi_i^p \; = \; E \zeta_i^p\, \prod_{j=1}^{i-1}E(1-\zeta_j)^p\; = \; \frac{1}{(1+p)^i},
\end{equation*}
hence
\begin{equation*}
  \| D(\xi)\|_p\; \le \; 2 +\sum_{i=1}^\infty i \|\xi_i\|_p \ < \ \infty\quad\text{for all } p>1.
\end{equation*}
Using similar arguments as in the proof of Lemma~\ref{lem:TPL3} we obtain that the series 
\begin{equation}\label{eq:Zinfty}
   Z_\infty:= \sum_{u\in\bV} X_\infty(A_u) \,D\bigl(\rho(X_\infty,u)\bigr)
\end{equation}
converges in $L^p$ for all $p>1$.
Further, for $\xi\sim\GEM(a)$ with $a=(a_1,\ldots,a_k)\in\bN^\star$ and $b:=a_1+\cdots+a_k$, 
Lemma~\ref{lem:GEMmargin} leads to
\begin{equation}\label{eq:condD}
  ED(\xi) \; = \; -2 \, +\, \frac{1}{1+b} \biggl(\sum_{i=1}^\infty i a_i \, +\, k \, +\, 2\biggr)
\end{equation}
if $k>0$, and $E D(\xi)=0$ for $\xi\sim \GEM$.

\begin{theorem}\label{thm:HPL} 
As $n\to\infty$,
\begin{equation*}
  \frac{1}{n}\,\HPL(X_n)-H_n  +2 \;\to\;  Y_\infty +  Z_\infty,  
\end{equation*}
almost surely and in $L^p$ for every $p>0$.  
\end{theorem}

\begin{proof} 
We project $Z_\infty$ on the natural filtration. Using~\eqref{eq:condD} we get
\begin{align*}
    E[Z_\infty|\cF_n] \ &=\ \sum_{u\in\bV} E[X_\infty(A_u)|\cF_n]\, E[D(\rho(X_\infty,u))|\cF_n]\\
           &=\ \sum_{u\in X_n} \frac{\# X_n(u)}{n} 
                       \biggl(-2+\frac{1}{\# X_n(u)} \Bigl(\sum_{i=1}^\infty i \,\#X_n(ui) \, +\, \#\{i\in\bN:\, ui\in X_n\} + 2\Bigr)\biggr)\\
           &=\ -\frac{2}{n}\sum_{u\in X_n} \#X_n(u) \; +\;\frac{1}{n}\sum_{u\in X_n} \sum_{i=1}^\infty i \,\#X_n(ui) \; 
                         +\; \frac{1}{n}\sum_{u\in X_n} \#\{i\in\bN:\, ui\in X_n\}\;+\; 2\\
           &=\ -\frac{2}{n}\, \TPL(X_n)\; +\; \frac{1}{n}\sum_{u\in X_n} |u|_1 \; +\; \frac{n-1}{n}\\
          &=\  -\frac{1}{n}\, \TPL(X_n) \; + \; \frac{1}{n}\, \HPL(X_n)\; +\; \frac{n-1}{n}.
\end{align*}
Now we proceed as in the proof of Theorem~\ref{thm:TPL}.
\end{proof}

As in the vertical case, see the remark after the proof of Theorem~\ref{thm:TPL}, we may use the calculations in the
proof to obtain an explicit formula for the mean horizontal path length,
\begin{equation}\label{eq:meanVertPL}
  E\HPL(X_n)\, =\,  -(n-1) + n\,E Z_\infty + E\TPL(X_n)\, =\, nH_n-2n+1\ \text{ for all } n\in\bN.
\end{equation}

\subsection{The Wiener index}\label{subsec:Wiener}

The chemist H.~Wiener introduced  
\begin{equation}\label{eq:WI}
   \WI(G) := \frac{1}{2}\sum_{(u,v)\in V\times V} d_\circ (u,v)
\end{equation}
as a measure of spread of an arbitrary finite connected graph $G$ with node set $V$. 
Here $d_\circ$ denotes the canonical graph distance, i.e.~$d_\circ (u,v)$ is the minimum length
of a path connecting $u$ and $v$ in $G$. Let $u\wedge v$ be the longest common prefix of
$u,v\in\bV$. For trees we then have
\begin{equation*}
  d_\circ (u,v) = |u|+|v| - 2 |u\wedge v|
\end{equation*}
and, as in the case of binary trees~\cite[eq.(34) corrected]{GrMtree},
\begin{equation*}
   \sum_{(u,v)\in x\times x} |u\wedge v| \; = \; \sum_{u\in x} \#x(u)^2\, -\, \#x^{\, 2},
\end{equation*}
so that we may rewrite the Wiener index for $x\in\bH_n$ in terms of total path length and subtree sizes as
\begin{equation*}
  \WI(x)\, = \; n\, \TPL(x) + n^2 - \sum_{u\in x} \# x(u)^2.
\end{equation*}
Again, we will show that a suitably standardized version converges almost surely if we insert for $x$ the random variables 
$X_n$ of the Harris chain. In addition to $Y_\infty$ as in~\eqref{eq:Yinfty} we need
\begin{equation*}
  W_\infty := \sum_{u\in\bV} X_\infty(A_u)^2.
\end{equation*}
Arguments similar to those used for $Y_\infty$ in the proof of Lemma~\ref{lem:TPL1} show that this series converges 
almost surely and that the limit has moments of all orders.

\begin{theorem}\label{thm:WI} 
As $n\to\infty$,
\begin{equation*}
  \frac{1}{n^2}\, \WI(X_n) -H_n  +1 \;\to\; Y_\infty-W_\infty,  
\end{equation*}
almost surely and in $L^p$ for every $p>0$.  
\end{theorem}

\begin{proof}
As in the proof of the corresponding results for the other tree functionals, we project the right hand side of the
formula on the natural filtration. For $Y_\infty$ this has been done in Section~\ref{subsec:TPL}. For $W_\infty$, we 
proceed as follows: For $u=(u_1,\ldots,u_k)$ and $i=1,\ldots,k$ let 
$\xi_i := \rho_{u_i}\bigl(X_\infty,(u_1,\ldots,u_{i-1})\bigr)$. 
Then, as in the proof of Corollary~\ref{cor:toThmconddistr}, $X_\infty(A_u)^2=\prod_{i=1}^k \xi_i^2$,
so that, using~\eqref{eq:mombeta} and the conditional independence from Theorem~\ref{thm:conddistr},
\begin{align*}
  E[X_\infty(A_u)^2|\cF_n]\ &=\ \prod_{i=1}^{k} E[\xi_i^2|\cF_n]\\
          &=\ \prod_{i=1}^k \,\frac{\# X_n((u_1,\ldots,u_i)) \bigl(1+\# X_n((u_1,\ldots,u_i))\bigr)}
                               {\big(1+\sum_{j=1}^\infty \#X_n((u_1,\ldots,u_{i-1},j))\bigr) 
                                              \big(2+\sum_{j=1}^\infty \#X_n((u_1,\ldots,u_{i-1},j))\bigr)} \\
          &=\ \prod_{i=1}^k \,\frac{\# X_n((u_1,\ldots,u_i)) \bigl(1+\# X_n((u_1,\ldots,u_i))\bigr)}
                                                {\# X_n((u_1,\ldots,u_{i-1})) \bigl(1+\#
                                                  X_n((u_1,\ldots,u_{i-1}))}\\
         &=\ \frac{\#X_n(u)\, (\#X_n(u)+1)}{n(n+1)}
\end{align*}
whenever $u\in X_n$. 

In order to deal with the nodes not in $X_n$ we use the operation $v\mapsto \bar v=:\phi(v)$ introduced in the
proof of Lemma~\ref{lem:TPL2}. Let
\begin{equation*}
  \partial X_n :=\{ v\notin X_n:\, \phi(v)\in X_n\}
\end{equation*} 
be the set of external nodes of $X_n$ and put
\begin{equation*}
  A_k(v) := \{ w\in\bV:\, \phi^k(w)=v\}, \ \ k\in \bN_0,
\end{equation*}
where $\phi^0(u):=u$. Clearly, $\# \partial X_n=n$, $\# A_k(v)=2^k$ and 
%\begin{equation*}
%  \bV \setminus X_n \, = \, \sum_{v\in \partial X_n} \sum_{k=0}^\infty A_k(v).
%\end{equation*}
$  \bV \setminus X_n= \sum_{v\in \partial X_n} \sum_{k=0}^\infty A_k(v)$.
With $v=(v_1,\ldots,v_k)\in\partial X_n$, $\xi_i:=\rho_{v_i}\bigl(X_\infty,(v_1,\ldots,v_{i-1})\bigr)$ 
and $\tilde v:=(v_1,\ldots,v_{k-1})$ we get%, as in the first part of the proof, 
\begin{equation*}
  E[X_\infty(A_v)^2|\cF_n]\ =\ \Bigl(\,\prod_{i=1}^{k-1} E[\xi_i^2|\cF_n]\Bigr)\, E[\xi_k^2|\cF_n]\
         =\ \frac{\#X_n(\tilde v)\, (\#X_n(\tilde v)+1)}{n(n+1)}\, E[\xi_k^2|\cF_n].
\end{equation*}
Conditionally on $\#X_n(\tilde v1)=a_1,\ldots,\#X_n(\tilde vj) =a_j$, $j:=v_k-1$,
the distribution of $\xi_k$ is equal to the distribution of $YZ$, with $Y,Z$ independent and 
\begin{equation*}
  Y\sim \Beta\Bigl(1,\sum_{j=1}^{k-1} a_j\Bigr), \quad Z\sim \unif(0,1).
\end{equation*}
In view of $\sum_{j=1}^{k-1} a_j=\#X_n(\tilde v)$ we thus obtain
\begin{equation*}
  E[\xi_k^2|\cF_n]\, =\, \frac{2}{\#X_n(\tilde v)\, (\#X_n(\tilde v)+1)}\cdot \frac{1}{3},
\end{equation*}
and hence, for $w\in A_k(v)$, 
\begin{equation*}
  E[X_\infty(A_w)^2|\cF_n] \, =\, \frac{2}{n(n+1)}\, \Bigl(\frac{1}{3}\Bigr)^{k+1}.
\end{equation*}
For the contribution of the nodes not in $X_n$ to the conditional expactation of $W_\infty$ this gives 
\begin{align*}
  \sum_{u\notin X_n} E[X_\infty(A_u)^2|\cF_n]\ 
        &=\ \sum_{v\in\partial X_n}  \sum_{k=0}^\infty \sum_{w\in A_k(v)}   E[X_\infty(A_w)^2|\cF_n] \\
        &=\ \sum_{v\in\partial X_n} \frac{2}{n(n+1)} \sum_{k=0}^\infty 2^k\Bigl(\frac{1}{3}\Bigr)^{k+1}.\\
        &=\ \frac{2}{n+1}.
\end{align*}
Putting pieces together we arrive at
\begin{equation*}
  E[W_\infty|\cF_n] \; = \; \frac{1}{n(n+1)}\biggl(n+\TPL(X_n) + \sum_{u\in X_n} \#X_n(u)^2\biggr) \; 
                                         +\; \frac{2}{n(n+1)},
\end{equation*}
and we can now proceed as in the proof of Theorem~\ref{thm:TPL}.
\end{proof}

Again, we can use the proof to obtain expected values,
\begin{equation*}
  E\WI(X_n)\, =\,  n(n+1)H_n-2n^2 \quad \text{for all } n\in\bN.
\end{equation*}
This agrees with Neininger's 
result~\cite[Theorem~1.2]{RRTWiener}.

\subsection{Distributional considerations}\label{subsec:distr}
Let $X_{\infty,i}$, $i\in\bN$, be as in the proof of Theorem~\ref{thm:lim2}.
For the total path length the representation $Y_\infty=\Phi(X_\infty)$ in Section~\ref{subsec:TPL} of the limit
$Y_\infty$ in terms of $X_\infty$ leads to
\begin{equation}\label{eq:rec0TPL}
  Y_\infty \; = \; C\bigl(\rho(X_\infty,\emptyset)\bigr) + \sum_{i=1}^\infty X_\infty(A_{(i)}) \, Y_{\infty,i},
\end{equation}
with $Y_{\infty,i}:= \Phi(X_{\infty,i})$. Note that this is an equality for random variables (strictly speaking, it
refers to the underlying probability measure as we may have to discard a null set for $X_\infty$ to be atom-free
and diffuse). In terms of distributions this may be rewritten as 
\begin{equation}\label{eq:rec1TPL} 
  Y_\infty \, \dequ \, C(\rho) + \sum_{i=1}^\infty \rho_i \, Y_\infty^{(i)},
\end{equation}
with $\rho,Y_\infty^{(1)}, Y_\infty^{(2)},\ldots$ independent, and $\rho\sim \GEM$, $Y_\infty^{(i)}\dequ Y_\infty$ for
all $i\in\bN$. We recall that the `toll function' $C:\Sigma_\infty :[-\infty,\infty)$ in this distributional fixed point
equation is given by
\begin{equation*}
  C\bigl((s_i)_{i\in\bN}\bigr) \, = \, 1 + \sum_{i=1}^\infty s_i\log s_i.
\end{equation*}
On the other hand, it is known \cite{RRTpathlength} that the limiting total path length also satisfies 
\begin{equation}\label{eq:rec2TPL}
  Y_\infty \, \dequ\,  UY_\infty + (1-U)Y_\infty^\star + G(U),
\end{equation}
with $G(u):= u + u\log u + (1-u)\log(1-u)$,
$U,Y_\infty,Y_\infty^\star$ independent, $U\sim\unif(0,1)$,  and $Y_\infty^\star\dequ Y_\infty$.  
What is the connection between the two equations?

Suppose that $\xi\sim\GEM$ and let $\zeta=(\zeta_i)_{i\in\bN}$ be related
to $\xi$ as in~\eqref{eq:xidef1}. Consider the shifted sequence $\tilde\zeta=(\tilde\zeta_i)_{i\in\bN}$ with 
$\tilde\zeta_i=\zeta_{i+1}$ for all $i\in\bN$. Clearly, $\tilde\zeta$ is again a sequence of independent, 
$\unif(0,1)$-distributed random variables, and it is independent of $\zeta_1$. This implies that the corresponding
$\tilde\xi$ is GEM distributed, and we have
\begin{align*}
  C(\xi) \ &=\ 1 + \zeta_1\log(\zeta_1) +
                                      (1-\zeta_1)\sum_{i=1}^\infty \tilde\xi_i \bigl(\log(1-\zeta_1) + \log\tilde\xi_i\bigr)\\
                &=\ \zeta_1 + \zeta_1\log(\zeta_1) + (1-\zeta_1) \log(1-\zeta_1) + (1-\zeta_1)\, C(\tilde\xi).
\end{align*}
Using~\eqref{eq:rec0TPL} we now get, with $\zeta_1=\rho_1(X_\infty,\emptyset)$ and 
$\tilde \xi_i=\rho_{i+1}(X_\infty,\emptyset)$,
\begin{equation*}
  Y_\infty\; =\; G(\zeta_1) + \zeta_1 Y_{\infty,1} 
                             + (1-\zeta_1)Y_\infty^\star, 
    \quad\text{with } \ Y_\infty^\star := C (\tilde\xi) + \sum_{i=1}^\infty \tilde\xi_i \, Y_{\infty,i+1}.
\end{equation*}
Together with~\eqref{eq:rec1TPL} this leads to the distributional equation~\eqref{eq:rec2TPL}.

It is instructive to compare this with a proof of~\eqref{eq:rec2TPL} that is based on the `musical decomposition' in
Section~\ref{subsec:treedecomp}.  The limit version of the decomposition given in Proposition~\ref{prop:decompinfty}
transforms $X_\infty$ into independent components $\eta=X_\infty(A_{(1)})$, $X_\infty^\flat$ and
$X_\infty^\sharp$ with the properties that
\begin{align*}
  \rho(X_\infty,A_{u^\flat})\; &=\; \rho(X_\infty^\flat,A_{u}) \ \text{ for all }u\in\bV,\\
   \rho(X_\infty,A_{u^\sharp})\; &=\; \rho(X_\infty^\sharp,A_{u})\ \text{ for all }u\in\bV, \, u\not=\emptyset,
\end{align*}
and with $\rho(X_\infty^\sharp,\emptyset)=\tilde\xi$, where $\tilde \xi$ is constructed from $\xi=\rho(X_\infty,\emptyset)$ 
as explained above. With this construction,
\begin{align*}
  \Phi(X_\infty) \; &=\; \sum_{u\in\bV} X_\infty(A_u)\, C(\rho(X_\infty,u))\\
                          &= \; X_\infty(A_\emptyset) \, C(\rho(X_\infty,\emptyset))
                               + \sum_{u\in\bV} X_\infty(A_{u^\flat})\, C(\rho(X_\infty,u^\flat))
                                +\! \sum_{u\in\bV,u\not=\emptyset} X_\infty(A_{u^\sharp})\, C(\rho(X_\infty,u^\sharp))\\
                          &=\; C(\xi) + \eta \sum_{u\in\bV} X^\flat_\infty(A_{u})\, C(\rho(X^\flat_\infty,u))
                                   +(1-\eta)\! \sum_{u\in\bV,u\not=\emptyset} X^\sharp_\infty(A_{u})\,
                                   C(\rho(X^\sharp_\infty,u))\\
                         &=\; C(\xi)-(1-\eta)\,C(\tilde\xi) \; + \; \eta\,\Phi(X^\flat_\infty)
                                     \;+\; (1-\eta)\,\Phi(X_\infty^\sharp), 
\end{align*}
and it remains to make use of $C(\xi)-(1-\eta)C(\tilde\xi)=G(\eta)$, which we have proved above. 
Once again, we note that the decomposition takes place on the level of the random quantities themselves;
there is no `$\dequ$'-sign.

As in the transition from Section~\ref{subsec:TPL} to Section~\ref{subsec:HTPL} the detailed consideration 
of the vertical case now makes it easy to treat the horizontal path length. With $\Psi(X_\infty)=Y_\infty+Z_\infty$
the limit in Theorem~\ref{thm:HPL} we just replace $C$ by $C+D$ to obtain the decomposition
\begin{equation*}
  \Psi(X_\infty) \; =\; (C+D)(\xi)-(1-\eta)\,(C+D)(\tilde\xi) \; + \; \eta\,\Psi(X^\flat_\infty)
                                     \;+\; (1-\eta)\,\Psi(X_\infty^\sharp).
\end{equation*}
A straightforward computation gives $D(\xi)-(1-\eta)\,D(\tilde\xi)=1-2\eta$, which leads to the horizontal analogue 
of~\eqref{eq:rec2TPL} with $\tilde G(u):= 1-u + u\log u + (1-u)\log(1-u)$ instead of $G$. Clearly, $\tilde G(\eta)$ 
and $G(\eta)$ are equal in distribution, which implies that the limit distributions arising in the vertical and
horizontal case satisfy the same fixed point equation. It is straightforward to set up a metric space of probability
distributions which contains these limit distributions and that turns the right hand side of~\eqref{eq:rec2TPL} into 
a contraction, hence the limit distributions arising for the vertical and horizontal path length of random recursive
trees are identical.

The above argument depends on the limit version of the decomposition. With some additional work
the finite version in Section~\ref{subsec:treedecomp} can be used directly to obtain the
convergence in distribution of the standardized path length; see~\cite{RoeQS} for the 
Quicksort situation. As pointed out at the beginning of this section, the contraction
method may miss the fact that the random variables themselves converge. On the other hand, 
as the above path length example shows,  the approach via a fixed point relation for the limit 
distribution may lead to the direct recognition of the equality of two limit distributions, which 
may not be apparent from the representation of the respective limit random variables in terms
of the limit tree (indeed, the representations $Y_\infty$ and $Y_\infty+Z_\infty$, given in 
Theorems~\ref{thm:TPL} and~\ref{thm:HPL} respectively, seem to suggest that the limit
distributions are different).

%\bigbreak
\begin{figure}%[h]
  \begin{center}
  \includegraphics[width=6cm]{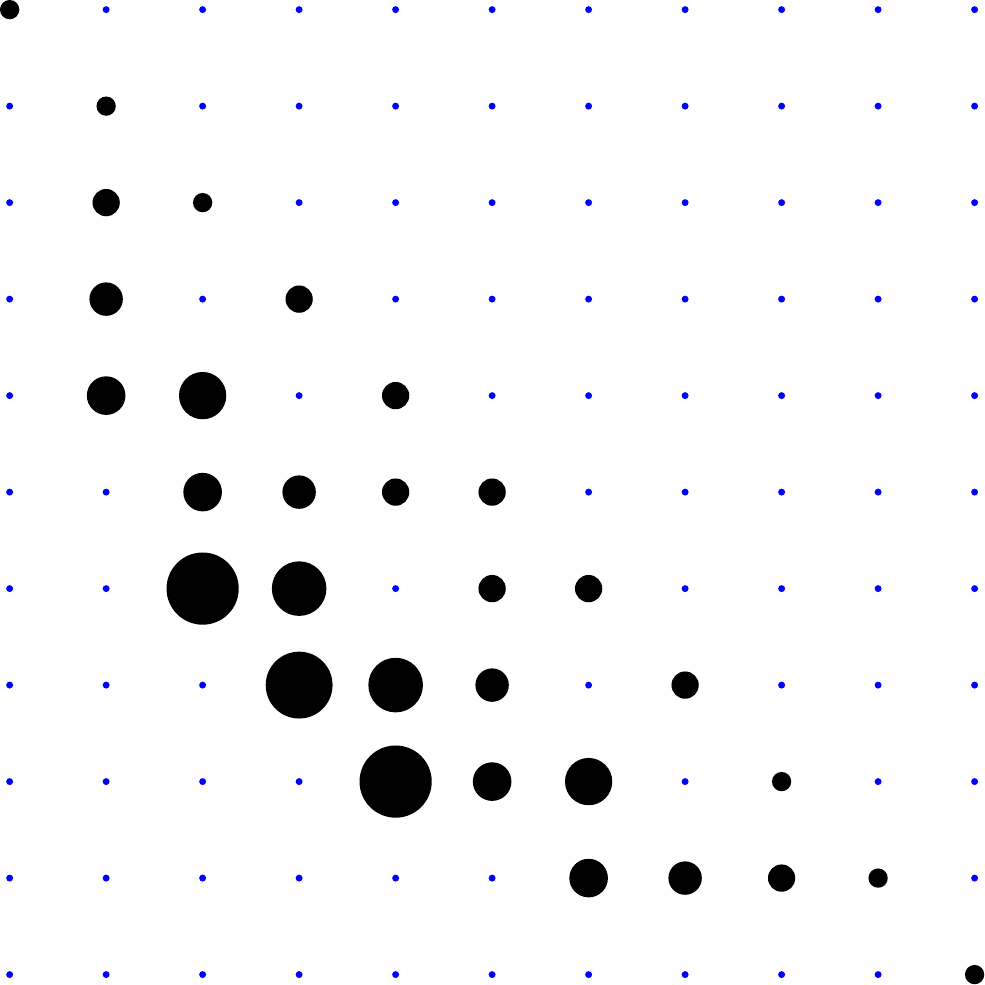}
 \end{center}
%\vspace{-.2cm}
\caption{Joint distribution of the vertical and horizontal total path length for the trees in $\bH_7$.}\label{fig:VHTPL}
\end{figure}

Equality of the limit distributions naturally raises the question whether there is a relation between the respective
distributions for finite trees.  Figure~\ref{fig:VHTPL} shows the pair $(i,j)$ of values $i$ for the vertical and $j$
for the horizontal total path length for all $6!=720$ recursive trees with $7$ nodes, where the sizes of the black dots
correspond to the multiplicities of the pairs and the blue dots represent pairs that do not appear. The picture suggests
that, up to a shift that is apparent from~\eqref{eq:meanVertPL}, 
the joint distribution of total vertical and total horizontal path length is symmetric.  Clearly,
this would imply that the limit distributions are the same.

We now define $T:\bV\to\bV\,$ by $T(\emptyset)=\emptyset$, $T((1))=(1)$ and, if $u=(u_1,\ldots,u_k)$ and $T(u)=v$ with
$v=(v_1,\ldots,v_j)$, by
\begin{equation}\label{eq:Tdefrec}
   \begin{split}
  T\bigl((u_1,\ldots,u_k,1)\bigr)\; &:=\; (v_1,\ldots,v_{j-1},v_j+1),\\
  T\bigl((u_1,\ldots,u_{k-1},u_k+1)\bigr)\; &:=\; (v_1,\ldots,v_j,1).
  \end{split}
\end{equation}
It is easy to see that $T$ is bijective; in fact, $T^{-1}=T$ ($T$ can be related to the natural correspondence mentioned
after the proof of Theorem~\ref{thm:lim1}; see~\cite{IMdiss}). The recursive part~\eqref{eq:Tdefrec} translates a move
downwards into a move to the right and vice versa. Further, $T$ is compatible with tree growth: If we add a node $u$ to a tree $x$ as a
first child of $v\in x$, then $T(u)$ is the next next child to the parent of $T(u)$ and, again, vice versa. In
particular, writing $T(x)$ for $\{T(u):\, u\in x\}$, we may lift $T$ to a bijective map on $\bH$ with the property that
$T(\bH_n)=\bH_n$ for all $n\in\bN$. This construction proves
\begin{equation*}
  \cL\bigl(\TPL(X_n)-(n-1)\bigr) \;=\; \cL\bigl(\HPL(X_n)\bigr) \ \text{ for all } n\in\bN, \, n\ge 2,
\end{equation*}
if we can show that the distribution of the Harris chain $(X_n)_{n\in\bN}$ is invariant under $T$ and that
\begin{equation}\label{eq:TPLHPL}
  \HPL\bigl(T(x)\bigr) = \TPL\bigl(x\bigr)-1\ \text{ for all } x\in\bH,\, \#x>1.
\end{equation}
The first of these is an immediate consequence of the tree growth mechanism. To obtain~\eqref{eq:TPLHPL} it is
enough to show that
\begin{equation*}
  |T(u)|_1-|T(u)| = |u|-1\ \text{ for all } u\in\bV,\, u\not=\emptyset.
\end{equation*}
This, however, can easily be proved by induction, considering the two cases in~\eqref{eq:Tdefrec} separately.

In view of this simple bijective proof one may naturally wonder what the advantage of the boundary theory approach
might be.  Almost sure convergence of the standardized vertical and horizontal path lengths implies the convergence of any linear
combinations, for example. This is of interest in connection with the analysis of the recursive tree algorithm $\RT$
introduced in Section~\ref{subsec:alg}: The number $C_n$ of comparisons needed to build the tree $X_n$ for $n-1$ data is given
by the sum of the horizontal and the vertical path length of $X_n$, hence
\begin{equation*}
  EC_n=2nH_n-3n-1, \quad \frac{1}{n}\bigl(C_n-EC_n)\to 2Y_\infty+Z_\infty\ \text{with probability } 1,
\end{equation*}
with $Y_\infty$ and $Z_\infty$ as in Sections~\ref{subsec:TPL} and~\ref{subsec:HTPL}.
While the mean can be obtained from the symmetry and the individual results for the two versions of path length, we
would  need their joint distribution in order to obtain the limit result for the sum.

\bibliographystyle{amsalpha}
\bibliography{rrt}

\end{document}